\documentclass[12pt]{article}
\usepackage{hyperref}
\usepackage{amsmath}
\usepackage[dvips]{graphicx}
\usepackage{amsthm}
\usepackage{epsfig}
\usepackage{amssymb}
\usepackage{amsfonts}

\newtheorem{theorem}{Theorem}[section]

\newtheorem{proposition}{Proposition}[section]

\catcode`@=11
\@addtoreset{equation}{section}
\catcode`@=12

\begin{document}

\begin{center}
{\Large  A Liouville theorem for an integral equation of the Ginzburg-Landau type}
\end{center}

\vskip 5mm

\begin{center}
{\sc Yutian Lei \quad and \quad Xin Xu} \\

\end{center}

\vskip 5mm {\leftskip5mm\rightskip5mm \normalsize
\noindent{\bf{Abstract}}
In this paper, we are concerned with a Liouville-type
result of the nonlinear integral equation
\begin{equation*}
u(x)=\overrightarrow{l}+C_*\int_{\mathbb{R}^{n}}\frac{u(1-|u|^{2})}{|x-y|^{n-\alpha}}dy.
\end{equation*}
Here $u: \mathbb{R}^{n} \to \mathbb{R}^{k}$ is a bounded, uniformly continuous and
differentiable function with $k \geq 1$ and $1<\alpha<n$,
$\overrightarrow{l} \in \mathbb{R}^{k}$ is a constant vector, and $C_*$ is a real constant.
If $u$ is the finite energy solution,
we prove that $|\overrightarrow{l}| \in \{0,1\}$. Furthermore, we also give a Liouville
type theorem (i.e., $u \equiv \overrightarrow{l}$).

}

\renewcommand{\theequation}{\thesection.\arabic{equation}}
\catcode`@=11
\@addtoreset{equation}{section}
\catcode`@=12

\vskip 1cm
\section{Introduction}  %%%%%%%%%%%%%%%%%%%%%%%%%%%%%%%%%%%%

We first recall several Liouville theorems.
If a harmonic function $u$ is bounded on $\mathbb{R}^n$, then $u \equiv Constant$.
When $\alpha \in (0,2)$, $u$ is a bounded function satisfying
$(-\Delta)^{\frac{\alpha}{2}}u=0$ on $\mathbb{R}^n$,
then $u \equiv Constant$ (cf. \cite{BKN}, \cite{ZCCY}).

In 1994, Brezis, Merle and Rivi\`{e}re \cite{BMR} studied the
quantization effects of the following equation
\begin{equation}\label{1.5}
-\Delta u=u(1-|u|^2) \quad on \ \mathbb{R}^2.
\end{equation}
Here $u:\mathbb{R}^2 \to \mathbb{R}^2$ is a vector valued function.
It is the Euler-Lagrange equation of the Ginzburg-Landau energy
$$
E_{GL}(u)=\frac{1}{2}\|\nabla
u\|_{L^2(\mathbb{R}^2)}^2+\frac{1}{4}\|1-|u|^2\|_{L^2(\mathbb{R}^2)}^2.
$$
In particular,
they proved that the finite energy solution (i.e., $u$ satisfies $\nabla
u \in L^2(\mathbb{R}^2)$) is bounded (see also \cite{HH})
$$
|u| \leq 1 \quad on \ \mathbb{R}^2.
$$
Based on this result, they used the Pohozaev identity
to obtain a Liouville type theorem for finite
energy solutions
\begin{proposition} \label{th1.1} (Theorem 2 in \cite{BMR})
Let $u:\mathbb{R}^2 \to \mathbb{R}^2$ be a classical solution of (\ref{1.5}).
If $\nabla u \in L^2(\mathbb{R}^2)$, then either $u \in L^2(\mathbb{R}^2)$ which implies
$u \equiv \overrightarrow{0}$, or $1-|u|^2 \in L^1(\mathbb{R}^2)$ which implies $u \equiv \overrightarrow{C}$ with
$|\overrightarrow{C}|=1$.
\end{proposition}

In addition, for the integral equation
\begin{equation}\label{1.1}
u(x)=\int_{\mathbb{R}^n}\frac{u(1-|u|^2)}{|x-y|^{n-\alpha}}dy,
\end{equation}
there holds the following Liouville theorem.
\begin{proposition} \label{th1.2} (Theorem 1 in \cite{LCL})
Assume that $u:\mathbb{R}^n \to \mathbb{R}^k$ is bounded and differentiable, and solves (\ref{1.1})
with $k \geq 1$ and $\alpha \in (1,n/2)$. If $u \in L^2(\mathbb{R}^n)$, then $u(x) \equiv \overrightarrow{0}$.
\end{proposition}

For the Lane-Emden type integral equations with the critical case,
the finite energy solutions can be classified (cf. \cite{CLO}).
In this paper, we consider finite energy solutions of the integral
equation involving the Riesz potential
\begin{equation}\label{ie}
u(x)=\overrightarrow{l}+C_*\int_{\mathbb{R}^{n}}\frac{u(1-|u|^{2})}{|x-y|^{n-\alpha}}dy,
\end{equation}
where $u: \mathbb{R}^{n} \to \mathbb{R}^{k}$ with $k \geq 1$ and $0<\alpha<n$,
$\overrightarrow{l} \in \mathbb{R}^{k}$ is a constant vector,
and $C_* \in \mathbb{R}$ is a constant.

Eq. (\ref{ie}) is associated with the fractional Ginzburg-Landau equation (cf. \cite{MS}, \cite{MR}
and \cite{TZ})
\begin{equation}\label{1.2}
(-\Delta)^{\frac{\alpha}{2}} u=(1-|u|^{2})u  \quad on \ \mathbb{R}^n.
\end{equation}
Here $u=(u_1,u_2,\cdots,u_k):\mathbb{R}^n \to \mathbb{R}^k$.
Write $u_i^+=\max\{u_i,0\}$ and $u_i^-=-\min\{u_i,0\}$. Then
$u_i^+,u_i^- \geq 0$ and $u_i=u_i^+-u_i^-$.

When $\alpha \geq 2$ is an even number, Theorem 3.21 in \cite{HK}
(see also Theorem 2.4 in \cite{CDM}) shows that if $u$ solves (\ref{1.2}),
there exist two constants $l_i^+,l_i^- \in \mathbb{R}$ such that
\begin{equation}\label{f2}
u_i^+(x)=l_i^+ +C_\alpha\int_{\mathbb{R}^n}\frac{u_i^+(y)(1-|u(y)|^2)}{|x-y|^{n-\alpha}}dy,
\end{equation}
$$
u_i^-(x)=l_i^- +C_\alpha\int_{\mathbb{R}^n}\frac{u_i^-(y)(1-|u(y)|^2)}{|x-y|^{n-\alpha}}dy.
$$
Here
$
C_\alpha=\Gamma(\frac{n-\alpha}{2})[2^\alpha \pi^{n/2} \Gamma(\alpha/2)]^{-1}.
$
Denote $l_i^+-l_i^-$ by $l_i$, then there holds
$$
u_i(x)=l_i +C_\alpha\int_{\mathbb{R}^n}\frac{u_i(y)(1-|u(y)|^2)}{|x-y|^{n-\alpha}}dy,
$$
which implies that $u$ also solves (\ref{ie}), where
$\overrightarrow{l}=(l_1,l_2,\cdots,l_k)$. For the Lane-Emden
equation, Chen, Li and Ou obtained an analogous result (cf.
\cite{CLO}).

When $\alpha \in (0,2)$, if $u$ solves (\ref{1.2}) and $|u| \leq 1$ on $\mathbb{R}^n$, then
\begin{equation}\label{f1}
(-\Delta)^{\frac{\alpha}{2}} u_i^+=(1-|u|^{2})u_i^+  \quad on \ \mathbb{R}^n.
\end{equation}
On the other hand, $C_\alpha|x|^{\alpha-n}$ is a fundamental solution of
$(-\Delta)^{\frac{\alpha}{2}}u=0$ (cf. Chapter 5 in \cite{Stein}), i.e.,
$$
(-\Delta)^{\frac{\alpha}{2}}(C_\alpha|x|^{\alpha-n})=\delta_o,
$$
where $\delta_o$ is the Dirac mass at the origin $o$. Therefore,
on $\mathbb{R}^n$ we have
$$\begin{array}{ll}
&(-\Delta)^{\frac{\alpha}{2}}[C_\alpha|x|^{\alpha-n}*(u_i^+(1-|u|^2))]\\[3mm]
&=\delta_x *[u_i^+(1-|u|^2)]=u_i^+(x)(1-|u(x)|^2).
\end{array}
$$
Combining with (\ref{f1}) yields an $\alpha$-harmonic equation
$$
(-\Delta)^{\frac{\alpha}{2}}[u_i^+-C_\alpha|x|^{\alpha-n}*(u_i^+(1-|u|^2))]=0
  \quad on \ \mathbb{R}^n.
$$
Since $u_i^+ \geq 0$ and $u$ is bounded, we can see that the $\alpha$-harmonic function
$$
u_i^+-C_\alpha|x|^{\alpha-n}*(u_i^+(1-|u|^2))
$$
has an upper bound. And hence it is
a constant (cf. \cite{BKN}, \cite{ZCCY}), which is denoted by $l_i^+$. Thus,
(\ref{f2}) is also true. By the same argument above, we also see that $u$ solves (\ref{ie}).

In this paper, we
expect to obtain the analogous results in Propositions \ref{th1.1} and \ref{th1.2}
for the finite energy solutions of (\ref{ie}).

First, if the finite energy solution $u$ is bounded, we determine
the value of $|\overrightarrow{l}|$.

\begin{theorem} \label{th1.3}
Assume that a uniformly continuous function $u:\mathbb{R}^n \to
\mathbb{R}^k$ solves (\ref{ie}) with $k \geq 1$ and $n \geq 2$. If
\begin{equation}\label{cond1}
|u| \leq 1 \quad on \ \mathbb{R}^{n};
\end{equation}
and
\begin{equation}\label{cond2}
\int_{\mathbb{R}^{n}}|u|^2(1-|u|^2) dx <\infty.
\end{equation}
Then, one of the following results holds true

(Rt-i)  $u \in L^2(\mathbb{R}^{n})$ and $|\overrightarrow{l}|=0$ if $\alpha \in (0,n/2)$;

(Rt-ii)  $1-|u|^2 \in L^1(\mathbb{R}^{n})$ and $|\overrightarrow{l}|=1$ if $\alpha \in (0,n)$.
\end{theorem}

Next, we have a Liouville theorem for finite energy solutions.

\begin{theorem} \label{th1.4} (Liouville theorem)
Under the same assumption of Theorem \ref{th1.3}, if $u$ is also
differentiable, then

(Rt-iii) when (Rt-i) happens and $\alpha \in (1,n/2)$, we have $u \equiv \overrightarrow{0}$;

(Rt-iv) when (Rt-ii) happens and $\alpha \in (1,n)$, we have $u \equiv \overrightarrow{l}$.
\end{theorem}

Finally, we give remarks on conditions (\ref{cond1}) and (\ref{cond2}).

\paragraph{Remark 1.1.} The condition (\ref{cond1}) shows the boundedness
which is a necessary condition in the Liouville theorem.
For (\ref{1.2}) with $\alpha=2$, (\ref{cond1}) was proved firstly by Brezis (cf.
\cite{Brezis}, where the author used the Keller-Osserman theory via Kato's inequality),
and independently by Ma (cf. \cite{Ma2}, where the maximum principle was applied).
When $\alpha \in (0,2)$ and $1-|u|^2 \in L^2(\mathbb{R}^{n})$,
Ma also proved (\ref{cond1}) by the Kato inequality (cf. \cite{Ma1}).

\paragraph{Remark 1.2.} Sometimes (\ref{cond2}) is called a finite energy
condition. For example, $u$ is a finite energy (i.e., $\nabla
u \in L^2(\mathbb{R}^{n})$) solution of (\ref{1.2}) with $\alpha=2$.
Multiplying (\ref{1.2}) by $u$ and integrating on $B_R:=B_R(0)$ yield
\begin{equation}\label{f3}
\int_{B_R}|u|^2(1-|u|^2)dx
=\int_{B_R}|\nabla u|^2dx-\int_{\partial B_R}u\partial_{\nu}u ds,
\end{equation}
where $\nu$ is the unit outwards norm vector on $\partial B_R$.
The Sobolev inequality implies $\nabla u \in L^2(\mathbb{R}^{n})
\Rightarrow u \in L^{2^*}(\mathbb{R}^{n})$. Therefore,
$$
R\int_{\partial B_R}(|u|^{2^*}+|\nabla u|^2) ds \to 0,
\quad when \ \ R=R_j \to \infty.
$$
Thus, by the H\"older inequality, when $R \to \infty$,
\begin{equation}\label{fs}
\begin{array}{ll}
&\displaystyle \left|\int_{\partial B_R}u\partial_{\nu}u ds\right|\\[3mm]
&\displaystyle \leq \left(R\int_{\partial B_R}|u|^{2^*}ds\right)^{\frac{1}{2^*}}
\left(R\int_{\partial B_R}|\nabla u|^{2}ds\right)^{\frac{1}{2}}
|\partial B_R|^{\frac{1}{2}-\frac{1}{2^*}}
R^{-\frac{1}{2}-\frac{1}{2^*}}\\[3mm]
& \to 0.
\end{array}
\end{equation}
Inserting this into (\ref{f3}) we see that
\begin{equation}\label{sf}
\|\nabla u\|_{L^{2}(\mathbb{R}^{n})}^2=
\int_{\mathbb{R}^{n}}|u|^2(1-|u|^2)dx,
\end{equation}
and hence (\ref{cond2}) holds true.

On the contrary, if (\ref{cond2}) holds, then $\nabla u \in L^{2}(\mathbb{R}^{n})$.
In fact, take $\zeta_R$ as the cut-off function satisfying $\zeta_R(x)=1$ when $|x|
\leq R/2$ and $\zeta_R(x)=0$ when $|x| \geq R$. Multiplying (\ref{1.2})
by $u\zeta_R^2$ and integrating on $B_R$, we get
$$
\int_{B_R} \zeta_R^2|\nabla u|^2dx
=-2\int_{B_R}(u\nabla u)\cdot(\zeta_R\nabla \zeta_R)dx
+\int_{B_R}\zeta_R^2|u|^2(1-|u|^2)dx.
$$
Using the Cauchy-Schwarz inequality and (\ref{cond2}), we have
$$
\int_{B_R} \zeta_R^2|\nabla u|^2dx \leq C,
$$
where $C>0$ is independent of $R$. Letting $R \to \infty$ yields $\nabla u \in L^{2}(\mathbb{R}^{n})$.
Moreover, by the same argument above, we also have (\ref{fs}), and
hence (\ref{sf}) still holds true.

\section{Proof of Theorem \ref{th1.3}}

\begin{theorem} \label{th2.1}
Under the same assumptions of Theorem \ref{th1.3}, we have

(i) either $u \in L^2(\mathbb{R}^{n})$ and $\lim_{|x| \to \infty}|u(x)|=0$;

(ii) or $1-|u|^2 \in L^1(\mathbb{R}^{n})$ and $\lim_{|x| \to \infty}|u(x)|=1$.
\end{theorem}

\begin{proof}
Here an idea in Section 3.2 of \cite{BMR} was used.

Set $S_*=\{x \in \mathbb{R}^n; \frac{1}{4} \leq |u| \leq \frac{3}{4}\}$.
We claim that there exists
suitably large $R_0>0$ such that $S_* \subset B_{R_0}(0)$.

Otherwise, we can find a sequence $\{x_j\} \subset S_*$ satisfies $\lim_{j \to
\infty}x_j=\infty$. Since $u$ is uniformly continuous, there exists
$\eta \in (0,1)$ such that
$$
\frac{1}{8} \leq |u(x)| \leq \frac{7}{8},
\quad when \ |x-x_j|<\eta, \ \forall j.
$$
Choose a subsequence of $\{x_j\}$ (still denoted by itself) satisfying
$|x_i-x_j|>3\eta$ for $i \neq j$. Therefore,
$$
\int_{\cup_j B(x_j,\eta)}|u|^2(1-|u|^2)dx \geq
C|\cup_j B(x_j,\eta)|=\infty,
$$
which contradicts with (\ref{cond2}). The claim is proved.

Since
$u$ is uniformly continuous and $\mathbb{R}^n\setminus B_{R_0}(0)$ is connected,
either $|u| \leq 1/4$ or $|u| \geq 3/4$ holds true on
$\mathbb{R}^n\setminus B_{R_0}(0)$.

When $|u| \leq 1/4$ on $\mathbb{R}^n \setminus B_{R_0}(0)$, it is led
to $1-|u|^2 \geq 15/16$ on $\mathbb{R}^n \setminus
B_{R_0}(0)$. Thus,
$$\begin{array}{ll}
&\displaystyle
\int_{\mathbb{R}^n}|u|^2dx
=\int_{B_{R_0}(0)}|u|^2dx+\int_{\mathbb{R}^n \setminus B_{R_0}(0)}|u|^2dx\\[3mm]
&\displaystyle \leq |B_{R_0}| +\frac{16}{15}\int_{\mathbb{R}^n
\setminus B_{R_0}(0)}|u|^2(1-|u|^2)dx <\infty,
\end{array}
$$
by virtue of (\ref{cond1}) and (\ref{cond2}). Namely, $u \in L^2(\mathbb{R}^n)$.

Now, we claim that
\begin{equation}\label{asy1}
\lim_{|x| \to \infty}|u(x)|=0.
\end{equation}

Otherwise, we can find $\epsilon_0>0$ and $x_j \to \infty$ such that
$$
|u(x_j)| \geq 2\epsilon_0.
$$
Since $u$ is uniformly continuous, there exists $\eta>0$, such that
$$
|u(x)-u(y)|<\epsilon_0, \quad when \ |x-y|<\eta.
$$
Therefore, for $x \in B_\eta(x_j)$, $|u(x)|>|u(x_j)|-\epsilon_0 \geq \epsilon_0$.
Thus,
$$
\int_{B_\eta(x_j)}|u(x)|^2dx>\epsilon_0^2|B_\eta|.
$$
This contradicts with $u \in L^2(\mathbb{R}^n)$.

When $|u| \geq 3/4$ on $\mathbb{R}^n\setminus B_{R_0}(0)$, by the same
argument above, we can see firstly
$$
\int_{\mathbb{R}^n}(1-|u|^2) dx <\infty
$$
in view of (\ref{cond1}) and (\ref{cond2}). This is $1-|u|^2
\in L^1(\mathbb{R}^n)$. Second, since $u$ is uniformly continuous, there also holds
\begin{equation}\label{asy2}
\lim_{|x| \to \infty}|u(x)|=1.
\end{equation}
The proof of Theorem \ref{th2.1} is complete.
\end{proof}

\begin{theorem} \label{th2.2}
Under the same assumptions of Theorem \ref{th1.3}, we have

(i) when $u \in L^2(\mathbb{R}^{n})$ and $\alpha \in (0,n/2)$, then $|\overrightarrow{l}|=0$;

(ii) when $1-|u|^2 \in L^1(\mathbb{R}^{n})$ and $\alpha \in (0,n)$, then $|\overrightarrow{l}|=1$.
\end{theorem}

\begin{proof}
Set
\begin{equation}\label{v-def}
v(x)=\frac{1}{n-\alpha}
\int_{\mathbb{R}^{n}}\frac{|u(y)|(1-|u(y)|^{2})}{|x-y|^{n-\alpha}}dy.
\end{equation}
By exchanging the integral variants, we can get
$$
v(x)=\int_0^{\infty}
\frac{\int_{B_t(x)}|u(y)|[1-|u(y)|^2]dy}{t^{n-\alpha}}
\frac{dt}{t}.
$$
Once there holds
\begin{equation}\label{rest}
\lim_{|x| \to \infty}v(x)=0,
\end{equation}
from (\ref{ie}) it follows that
$$
\lim_{|x| \to \infty}u(x)=\overrightarrow{l}.
$$
By (\ref{asy1}) and (\ref{asy2}), the proof of Theorem \ref{th2.2} is complete.

\textbf{Proof of (\ref{rest}).}
Take $x_0 \in \mathbb{R}^n$. By (\ref{cond1}), we know that
$\forall \varepsilon>0$, there exists
$\delta \in (0,1)$ such that
$$
\int_0^{\delta}
\frac{\int_{B_t(x_0)}|u(z)|[1-|u(z)|^2]dz}{t^{n-\alpha}}
\frac{dt}{t} \leq C
\int_0^{\delta} t^{\alpha} \frac{dt}{t}
<\varepsilon.
$$
When $|x-x_0|<\delta$, $B_t(x_0) \subset B_{t+\delta}(x)$. Therefore,
$$\begin{array}{ll}
&\displaystyle \int_{\delta}^{\infty}
\frac{\int_{B_t(x_0)}|u(z)|[1-|u(z)|^2]dz}{t^{n-\alpha}}
\frac{dt}{t}\\[3mm]
& \leq C\displaystyle\int_{\delta}^{\infty}
\frac{\int_{B_{t+\delta}(x)}|u(z)|[1-|u(z)|^2]dz}{(t+\delta)^{n-\alpha}}
(\frac{t+\delta}{t})^{n-\alpha+1} \frac{d(t+\delta)}{t+\delta}\\[3mm]
&\displaystyle \leq C2^{n-\alpha+1}\int_0^{\infty}
\frac{\int_{B_t(x)}|u(z)|[1-|u(z)|^2]dz}{t^{n-\alpha}}
\frac{dt}{t}\\[3mm]
&\leq Cv(x).
\end{array}
$$
Combining two estimates above, we get
$$
v(x_0)<\varepsilon+Cv(x), \quad for \quad |x-x_0| <\delta.
$$
Thus, for any $s>1$,
\begin{equation}
\begin{array}{ll}
v^{s}(x_0)
&=|B_\delta(x_0)|^{-1}\displaystyle\int_{B_\delta(x_0)}
v^{s}(x_0)dx\\[3mm] &\leq
C\varepsilon^{s}
+C|B_\delta(x_0)|^{-1}\displaystyle\int_{B_\delta(x_0)}v^{s}(x)dx.
\end{array}
\label{4.15}
\end{equation}

We claim that
\begin{equation}\label{v-int}
v \in L^{s}(\mathbb{R}^n) \ for \  some \ s>1.
\end{equation}

In fact, using the Hardy-Littlewood-Sobolev inequality to (\ref{v-def})
yields
$$
\|v\|_{L^s(\mathbb{R}^n)} \leq
C\|u(1-|u|^2)\|_{L^{\frac{ns}{n+s\alpha}}(\mathbb{R}^n)}.
$$
When $u \in L^2(\mathbb{R}^n)$, we can choose $s=\frac{2n}{n-2\alpha}$ (in view
of $0<\alpha<n/2$).
When $1-|u|^2 \in L^1(\mathbb{R}^n)$, we can choose some $s>\frac{n}{n-\alpha}$.
Thus, (\ref{v-int}) is easy to prove by (\ref{cond1}).

In view of (\ref{v-int}), when $|x_0| \to \infty$,
$$
\int_{B_\delta(x_0)}v^{s}(x)dx \to 0.
$$
Inserting this result into (\ref{4.15}), we have
$$
\lim_{|x_0| \to \infty}v^{s}(x_0)=0.
$$
This result means that (\ref{rest}) holds.
\end{proof}

\section{Proof of Theorem \ref{th1.4}}
\begin{proof}
\textbf{Proof of (Rt-iii).}
When (Rt-i) holds true, $u$ solves
$$
u(x)=C_*\int_{\mathbb{R}^{n}}\frac{u(1-|u|^{2})}{|x-y|^{n-\alpha}}dy.
$$
By the same proof of Proposition \ref{th1.2}, we also obtain $u \equiv \overrightarrow{0}$.

\textbf{Proof of (Rt-iv).}

For convenience, we denote $B_R(0)$ by $B_R$ here.

{\it Step 1.} We claim that the  improper integral
\begin{equation}\label{jia}
\int_{\mathbb{R}^{n}}\frac{z\cdot\nabla[u(z)(1-|u(z)|^{2})]}{|x-z|^{n-\alpha}}dz
\end{equation}
is convergent at each $x \in \mathbb{R}^n$.

In fact, since $1-|u|^2 \in L^1(\mathbb{R}^n)$, we can find $R=R_j \to \infty$
such that
\begin{equation}\label{yi}
R\int_{\partial B_{R}}(1-|u(z)|^{2})ds \to 0.
\end{equation}
By (\ref{cond1}), we obtain that for sufficiently large $R$, there holds
$$
R\left|\int_{\partial B_R} \frac{u(z)(1-|u(z)|^{2})}{|x-z|^{n-\alpha}}ds \right|
\leq CR^{1-n+\alpha} \int_{\partial B_R}(1-|u(z)|^2)ds.
$$
Letting $R=R_j \to \infty$ and using (\ref{yi}) we get
\begin{equation}\label{bing}
R\int_{\partial B_R} \frac{u(z)(1-|u(z)|^{2})}{|x-z|^{n-\alpha}}ds
\to 0
\end{equation}
when $R=R_j \to \infty$.

Next, we claim that the improper integral
\begin{equation}\label{ji}
I(\mathbb{R}^n):=\int_{\mathbb{R}^n} \frac{u(z)(1-|u(z)|^{2})(x-z)\cdot z}
{|x-z|^{n-\alpha+2}}dz
\end{equation}
absolutely converges for each $x \in \mathbb{R}^n$.

In fact, we observe that the defect points of $I(\mathbb{R}^n)$ are $x$ and $\infty$.
When $z$ is near $\infty$, we have
$$\begin{array}{ll}
|I(\mathbb{R}^n\setminus B_r)|
&\displaystyle \leq C\int_{\mathbb{R}^n\setminus B_r}\frac{1-|u(z)|^2}{|x-z|^{n-\alpha}}dz\\[3mm]
&\displaystyle \leq C\left(\int_{\mathbb{R}^n}(1-|u|^2)^sdz\right)^{\frac{1}{s}}
\left(\int_r^\infty \rho^{n-\frac{s}{s-1}(n-\alpha)} \frac{d\rho}{\rho}\right)^{1-\frac{1}{s}}.
\end{array}
$$
Here $s \in (1,n/\alpha)$.
In view of $1-|u|^2 \in L^1(\mathbb{R}^n)$ and (\ref{cond1}), we get $1-|u|^2 \in L^s(\mathbb{R}^n)$
for all $s \geq 1$. Therefore,
\begin{equation}\label{xin}
|I(\mathbb{R}^n\setminus B_r)|<\infty.
\end{equation}
When $z$ is near $x$, by (\ref{cond1}) and $\alpha>1$,
$$
|I(B_\delta(x))| \leq
C\int_{B_\delta(x)}\frac{dz}{|x-z|^{n-\alpha+1}} \leq
C\int_0^\delta \rho^{\alpha-1} \frac{d\rho}{\rho}<\infty.
$$
Combining this with (\ref{xin}), we prove that (\ref{ji}) is absolutely convergent.

Finally we prove that (\ref{jia}) is convergent. Integrating by parts yields
\begin{equation}\label{geng}
\begin{array}{ll}
&\displaystyle\int_{B_R}\frac{z\cdot\nabla[u(z)(1-|u(z)|^{2})]}{|x-z|^{n-\alpha}}dz\\[3mm]
&\displaystyle =R\int_{\partial B_R} \frac{u(z)(1-|u(z)|^{2})}{|x-z|^{n-\alpha}}ds\\[3mm]
&\quad \displaystyle -n\int_{B_R}\frac{u(z)(1-|u(z)|^2)}{|x-z|^{n-\alpha}}dz\\[3mm]
&\quad \displaystyle -(n-\alpha)\int_{B_R} \frac{u(z)(1-|u(z)|^{2})(x-z)\cdot z}{|x-z|^{n-\alpha+2}}dz.
\end{array}
\end{equation}
Letting $R=R_j \to \infty$ in (\ref{geng}) and using (\ref{ie}) and (\ref{bing}),
we can see that
$$
\int_{\mathbb{R}^n}\frac{z\cdot\nabla[u(z)(1-|u(z)|^{2})]}{|x-z|^{n-\alpha}}dz
=-\frac{n}{C_*}\left(u(x)-\overrightarrow{l}\right)+(\alpha-n)I(\mathbb{R}^n),
$$
and hence it is convergent at each $x \in \mathbb{R}^n$.

{\it Step 2.} Proof of (Rt-iv).

For any $\lambda>0$, from (\ref{ie}) it follows
$$
u(\lambda x)=\overrightarrow{l}+C_*\lambda^{\alpha}\int_{\mathbb{R}^{n}}\frac{u(\lambda z)
(1-|u(\lambda z)|^{2})}{|x-z|^{n-\alpha}}dz.
$$
Differentiating both sides with respect to $\lambda$ yields
$$\begin{array}{ll}
&\quad x\cdot\nabla u(\lambda x)\\[3mm]
&=\displaystyle C_*\alpha\lambda^{\alpha-1}\int_{\mathbb{R}^{n}}\frac{u(\lambda z)
(1-|u(\lambda z)|^{2})}{|x-z|^{n-\alpha}}dz\\[3mm]
&\quad +\displaystyle C_*\lambda^{\alpha}\int_{\mathbb{R}^{n}}\frac{(z\cdot \nabla u(\lambda z))
(1-|u(\lambda z)|^{2})+u(\lambda z)
[-2u(\lambda z)(z\cdot \nabla  u(\lambda z))]}{|x-z|^{n-\alpha}}dz.
\end{array}
$$
Letting $\lambda=1$ yields
\begin{equation}\label{d}
x\cdot\nabla u(x)
=\alpha \left(u(x)-\overrightarrow{l}\right)
+C_*\int_{\mathbb{R}^{n}}\frac{z\cdot\nabla[u(z)(1-|u(z)|^{2})]}{|x-z|^{n-\alpha}}dz.
\end{equation}

Integrating by parts, we get
$$\begin{array}{ll}
&\displaystyle \int_{B_R}u(1-|u|^{2})(x\cdot\nabla u)dx\\[3mm]
&\displaystyle =-\frac{1}{4} \int_{B_R} x\cdot \nabla[(1-|u|^2)^2]dx\\[3mm]
&\displaystyle =\frac{n}{4}\int_{B_R}(1-|u|^2)^2dx
-\frac{R}{4}\int_{\partial B_{R}}(1-|u|^{2})^2 ds.
\end{array}
$$
Since (\ref{cond1}) and (\ref{yi}), it
follows that
\begin{equation}\label{ren}
R\int_{\partial B_{R}}(1-|u|^{2})^2 ds \to 0
\end{equation}
for some $R=R_j \to \infty$. Thus, by virtue of
(\ref{cond1}) and $1-|u|^2 \in L^1(\mathbb{R}^{n})$,
\begin{equation}\label{gui}
\int_{\mathbb{R}^{n}}u(x)(1-|u(x)|^{2})(x\cdot\nabla u(x))dx
=\frac{n}{4}\int_{\mathbb{R}^{n}}(1-|u(x)|^2)^2dx<\infty.
\end{equation}

From (\ref{cond1}) and $1-|u|^2 \in L^1(\mathbb{R}^{n})$, it
also follows that
$$
\left|\int_{\mathbb{R}^{n}}[u(x)\cdot
(u(x)-\overrightarrow{l})](1-|u(x)|^{2})dx\right| < \infty.
$$
Multiply (\ref{d}) by $u(x)(1-|u(x)|^{2})$ and integrate over
$B_R$. Letting $R=R_j \to \infty$, from the result above and
(\ref{gui}), we get
$$
\left|
\int_{\mathbb{R}^{n}}u(x)(1-|u(x)|^{2})
\int_{\mathbb{R}^{n}}\frac{z\cdot\nabla[u(z)(1-|u(z)|^{2})]}{|x-z|^{n-\alpha}}dzdx
\right| <\infty,
$$
and
\begin{equation}\label{g}
\begin{array}{ll}
&\quad\displaystyle\int_{\mathbb{R}^{n}}u(x)(1-|u(x)|^{2})(x\cdot\nabla u(x))dx\\[3mm]
&\quad -\displaystyle\alpha\int_{\mathbb{R}^{n}}[u(x)\cdot
(u(x)-\overrightarrow{l})](1-|u(x)|^{2})dx\\[3mm]
&=\displaystyle C_*\int_{\mathbb{R}^{n}}u(x)(1-|u(x)|^{2})
\int_{\mathbb{R}^{n}}\frac{z\cdot\nabla[u(z)(1-|u(z)|^{2})]}{|x-z|^{n-\alpha}}dzdx.
\end{array}
\end{equation}

We use the Fubini theorem and (\ref{ie}) to handle the term of the right hand
side. Thus,
\begin{equation}\label{g1}
\begin{array}{ll}
&\quad \displaystyle C_*\int_{\mathbb{R}^{n}}u(x)(1-|u(x)|^{2})\int_{\mathbb{R}^{n}}
\frac{z\cdot\nabla[u(1-|u|^{2})]}{|x-z|^{n-\alpha}}dzdx\\[3mm]
&=\displaystyle C_*\int_{\mathbb{R}^{n}}z\cdot\nabla[u(z)(1-|u(z)|^{2})]
\int_{\mathbb{R}^{n}}\frac{u(x)(1-|u(x)|^{2})}{|x-z|^{n-\alpha}}dxdz\\[3mm]
&=\displaystyle\int_{\mathbb{R}^{n}}(x\cdot\nabla[u(x)(1-|u(x)|^{2})])(u(x)-\overrightarrow{l})dx\\[3mm]
&=\displaystyle\int_{\mathbb{R}^n} (u(x)-\overrightarrow{l})(1-|u(x)|^2)](x\cdot \nabla u(x))dx\\[3mm]
&\quad \displaystyle +\int_{\mathbb{R}^n}u(x)(u(x)-\overrightarrow{l})[x\cdot \nabla(1-|u(x)|^2)]dx.
\end{array}
\end{equation}
Inserting this result into (\ref{g}), we
have
\begin{eqnarray*}
0&=&\alpha \int_{\mathbb{R}^n}u(x)(u(x)-\overrightarrow{l})(1-|u(x)|^2) dx\\
&&-\int_{\mathbb{R}^n} \overrightarrow{l}(1-|u(x)|^2)(x\cdot \nabla u(x))dx\\
&&-\int_{\mathbb{R}^n}u(x)(u(x)-\overrightarrow{l})[x\cdot \nabla(|u(x)|^2-1)]dx\\
&:=&I+II+III.
\end{eqnarray*}

We deal with the first and the third terms of the right hand side.

Noting $|\overrightarrow{l}|=1$, we have
\begin{eqnarray*}
I&=&\alpha \int_{\mathbb{R}^n}(u+\overrightarrow{l})(u-\overrightarrow{l})(1-|u|^2)dx\\
&&-\alpha \int_{\mathbb{R}^n}\overrightarrow{l}(u-\overrightarrow{l})(1-|u|^2)dx\\
&=&-\alpha \int_{\mathbb{R}^n}(|u|^2-1)^2dx+\alpha \int_{\mathbb{R}^n}(1-|u|^2)dx\\
&&-\alpha \int_{\mathbb{R}^n}\overrightarrow{l}u(1-|u|^2)dx.
\end{eqnarray*}
Next, integrating by parts, we obtain
\begin{eqnarray*}
III&=&-\int_{\mathbb{R}^n}(u+\overrightarrow{l})(u-\overrightarrow{l})[x\cdot \nabla(|u|^2-1)]dx\\
&&+\int_{\mathbb{R}^n}\overrightarrow{l}(u-\overrightarrow{l})[x\cdot \nabla(|u|^2-1)]dx\\
&=& -\int_{\mathbb{R}^n}(|u|^2-1)[x\cdot \nabla(|u|^2-1)]dx\\
&&-\overrightarrow{l}\int_{\mathbb{R}^n}(|u|^2-1)\nabla \cdot[x (u-\overrightarrow{l})]dx\\
&=&\frac{n}{2}\int_{\mathbb{R}^n} (|u|^2-1)^2dx+ \overrightarrow{l}\int_{\mathbb{R}^n}(1-|u|^2)(x\cdot \nabla u)dx\\
&&+n\overrightarrow{l}\int_{\mathbb{R}^n}(u-\overrightarrow{l})(1-|u|^2)dx.
\end{eqnarray*}

Substituting these results into $I+II+III=0$, we get
\begin{eqnarray*}
&&(\alpha-\frac{n}{2})\int_{\mathbb{R}^n} (|u|^2-1)^2dx\\
&=& \int_{\mathbb{R}^n}(-\alpha \overrightarrow{l}u+\alpha+n\overrightarrow{l}u-n)(1-|u|^2)dx\\
&=&\int_{\mathbb{R}^n}(n-\alpha )(\overrightarrow{l} u-1)(1-|u|^2)dx\\
&\leq&(n-\alpha )\int_{\mathbb{R}^n}(|u|-1)(1-|u|^2)dx.
\end{eqnarray*}
In view of $|u|-1 \leq \frac{1}{2}(|u|^2-1)$, it follows that
$$
(\alpha-\frac{n}{2})\int_{\mathbb{R}^n} (|u|^2-1)^2dx
\leq \frac{n-\alpha}{2} \int_{\mathbb{R}^n}(|u|^2-1)(1-|u|^2)dx,
$$
which implies
$|u| \equiv 1$ a.e. on $\mathbb{R}^n$. Inserting this into \eqref{ie},
we see that $u \equiv \overrightarrow{l}$ and hence (Rt-iv) is proved.
\end{proof}

\paragraph{Acknowledgements.}
%The authors thank Prof. P. Mironescu
%and the unknown referee very much for useful suggestions.
This research was supported by NNSF (11871278) of China.

%%%%%%%%%%%%%%%%%%%%%%%%%%%%%%%%%%%%%%%%%%%%%%%%%%%%%%%%%%%%%%%%%%%%%%%%%%%%%%%%%%

Yutian Lei

Jiangsu Key Laboratory for NSLSCS,
School of Mathematical Sciences,

Nanjing Normal University,
Nanjing, 210023, China

\vskip 3mm

Xin Xu

Depart. of Math., Southern University of Science and Technology,

Shenzhen, 518055, China

Faculty of Science and Technology, University of Macau, Macau,
China

\end{document}